\documentclass[12pt]{article}
\usepackage{amsmath,amsthm,amssymb,amsfonts, fancyhdr, color, comment, graphicx, environ,mathtools }
\usepackage{authblk}
\usepackage[sort]{natbib}
\setcitestyle{authoryear,open={(},close={)}} \usepackage[margin=1in]{geometry}

\usepackage{xcolor}
\newenvironment{keywords}{\textbf{\textit{Keywords---}}\begin{small}}{\end{small}}

\title{Asymptotic confidence sets for random linear programs}
\usepackage{times}
\usepackage[capitalize]{cleveref}
\usepackage{booktabs}
\usepackage{amsmath,amssymb,mathtools}
\usepackage{xifthen}
\usepackage{dsfont}

\newcommand{\cC}{\mathcal{C}}

\newcommand{\cI}{\mathcal{I}}

\newcommand{\cN}{\mathcal{N}}

\newcommand{\GG}{\mathbb{G}}

\newcommand{\NN}{\mathbb{N}}

\newcommand{\RR}{\mathbb{R}}

\newcommand{\bm}{\mathbf}

\newcommand*{\kl}[3][]{\ifthenelse{\isempty{#1}}{\operatorname{D}(#2\,\|\,#3)}{\operatorname{D}(#2\,\|\,#3\mid#1)}}

\DeclarePairedDelimiter{\triplenorm}{\vert\kern-0.25ex\vert\kern-0.25ex\vert}{\vert\kern-0.25ex\vert\kern-0.25ex\vert}

\newcommand*{\p}[1]{\mathbb P\left\{#1\right\}}

\newcommand*{\defeq}{\coloneqq}

\DeclareMathOperator*{\argmin}{argmin}

 \newtheorem{assumption}{Assumption}
\newtheorem{definition}{Definition}
\newtheorem{theorem}{Theorem}
\newtheorem{corollary}[theorem]{Corollary}
\newtheorem{lemma}[theorem]{Lemma}
\newtheorem{proposition}[theorem]{Proposition}

\newcommand*{\weakto}{\overset{D}\to}
\newcommand*{\bb}{\bm{b}}
\newcommand*{\bc}{\bm{c}}
\newcommand*{\bx}{\bm{x}}
\newcommand*{\bA}{\bm{A}}
\newcommand*{\bz}{\bm{0}}
\newcommand*{\val}{f}
\newcommand*{\optvert}{\bm{V}^*}
\newcommand*{\I}{\mathrm{I}}
\newcommand*{\bp}{\bm{p}}
\newcommand*{\proj}{\bar \bx^*_n}
\newcommand*{\supp}{\mathsf{h}}
\newcommand*{\haus}{\rho_H}
\newcommand*{\limfunc}[1]{\mathsf{g}_{#1}}

\author[1]{Shuyu Liu}
\author[2]{Florentina Bunea}
\author[1,3]{Jonathan Niles-Weed}
\affil[1]{Courant Institute of Mathematical Sciences, NYU}
\affil[2]{Department of Statistics and Data Science, Cornell University}
\affil[3]{Center for Data Science, NYU}

\begin{document}

\maketitle

\begin{abstract}Motivated by the statistical analysis of the discrete optimal transport problem, we prove distributional limits for the solutions of linear programs with random constraints.
Such limits were first obtained by Klatt, Munk, \& Zemel (2022), but their expressions for the limits involve a computationally intractable decomposition of $\RR^m$ into a possibly exponential number of convex cones.
We give a new expression for the limit in terms of auxiliary linear programs, which can be solved in polynomial time.
We also leverage tools from random convex geometry to give distributional limits for the entire set of random optimal solutions, when the optimum is not unique.
Finally, we describe a simple, data-driven method to construct asymptotically valid confidence sets in polynomial time.
\end{abstract}


\section{Introduction}
Linear programming is one of the core techniques in convex optimization, capturing many canonical problems such as maximum flow, shortest path, bipartite matching, and optimal transport.
Linear programs (LPs) are notable for their versatility, their rich combinatorial theory, and their algorithmic tractability: the pioneering work of \citet{Hac79} showed that LPs can be solved in polynomial time, and the last 70 years of research in theoretical computer science and scientific computing have made solving linear programs a ``mature technology'' in practice \citep{boyd2004convex}.

We consider throughout a standard form LP, given by
\begin{equation}\label{primal}
	\min_{\bm{x}\in\mathbb{R}^m} \langle\bm{c},\bm{x}\rangle,\qquad \text{s.t.}\ \bm{Ax}=\bm{b},\ \bm{x}\geq\bm{0},
\end{equation}
where $\bm{A}\in \mathbb{R}^{k\times m}$, $\bm{b}\in \mathbb{R}^k$ and $\bm{c}\in \mathbb{R}^m$.
The goal of this paper is to understand the distributional behavior of solutions to \cref{primal} when $\bb$ is replaced by a random vector $\bb_n$.
We assume the existence of a random variable $\GG$ such that
\begin{equation}
	r_n (\bb_n - \bb) \weakto \GG
	\label{b_limit}
\end{equation}
for some rate $r_n \to \infty$, and we will seek a corresponding limit law for the solutions to~\cref{primal}.
This setting is motivated by applications of linear programming in statistics and machine learning, where the ``right-hand side'' vector $\bb$ corresponds to random capacities, demands, or prices.
An important example, which motivates many of the developments of this paper, is the linear programming formulation of the optimal transportation problem between discrete distributions, where the vector $\bb$ corresponds to the probability mass function of the two measures.
The statistician who only has access to these measures via samples can compute a solution to an \emph{empirical optimal transport} problem by replacing $\bb$ with an estimator $\bb_n$.
Quantifying the uncertainty in the resulting solution requires constructing an asymptotic confidence set for this random linear program.

Obtaining distributional limit results for solutions to random optimization problems is, of course, a well studied subject both in scientific computing and in statistics~\citep{shapiro1991asymptotic, linderoth2006empirical, PolJud92,DupWet88,KinRoc93}, but the LP lacks the regularity conditions necessary to apply classical results: neither smoothness nor strong convexity holds for~\cref{primal} in general, solutions are generally not unique, and optimal solutions to \cref{primal} always lie on the boundary of the feasible set.
By contrast, standard distributional limit results, for instance in the analysis of M-estimators, require local strong convexity, uniqueness, and that the solution to the population-level problem lies in the relative interior of the feasible set \citep[see, e.g.,][]{vaart_1998}.
The challenges met in circumventing these classical conditions are well known \citep{MOMBound,Che54,AitSil58}.
Statistically, the lack of regularity in~\cref{primal} is the source of several pathologies: even when the solution to \cref{primal} is unique, the limiting distribution will in general not be Gaussian, and if there are multiple solutions to \cref{primal} it is not even clear how to formulate the desired distributional limit results.
The typical path forward, not taken in this work, is to impose extra conditions to guarantee that uniqueness holds and to focus on settings where there is sufficient regularity to ensure a Gaussian limit.

Let us give a very simple example which illustrates some of the difficulties of this problem.
Consider a $2 \times 2$ optimal transport problem:
\begin{align*}
	\min_{\pi \in \RR^{2 \times 2}} \pi_{12} + \pi_{21}\,, \qquad & \textrm{s.t.}\  \pi \bm{1} = \bm{r}, \pi^\top \bm{1} = \bm{s}, \pi \geq \bz\,,
\end{align*}
where $\bm{r} = \bm{s} = (1/2, 1/2)$.
In this case, the target solution $\pi^* = (1/2, 0; 0, 1/2)$ is unique.
If we suppose that $\bm{r}$ is replaced by random vector in the probability simplex $\bm{r}_n = (r^{(1)}_n, r^{(2)}_n)$, then the optimal solution to the perturbed program is
\begin{equation*}
	\hat \pi_n = (1/2, r^{(1)}_n - 1/2; 0, r^{(2)}_n) \mathds{1}_{\{ r_{n}^{(1)} > r_n^{(2)} \}} + (r^{(1)}_n, 0; r^{(2)}_n - 1/2, 1/2)\mathds{1}_{\{ r_{n}^{(1)} \leq r_n^{(2)} \}}\,,
\end{equation*}
and if we assume $\sqrt n (\bm{r}_n - \bm{r})$ converges in distribution to a centered Gaussian vector, the rescaled solution $\sqrt n(\hat \pi_n - \pi^*)$ converges to a mixture distribution with two non-Gaussian components.
A $3 \times 3$ version of the same problem, with the same objective function and $\bm{r} = \bm{s} = (1/3, 1/3, 1/3)$, has multiple optimal solutions, and \emph{a priori} it is not clear how to quantify the uncertainty of a solution obtained when $\bm{r}$ is replaced by a random counterpart.

The challenges in obtaining distributional limits for LPs were first tackled by the pioneering work of \cite{klatt2022limit}, who derived distributional limits for \eqref{primal} in a very general setting.
Their results are expressed in terms of a partition of $\RR^m$ into closed convex cones; the restriction of the limiting distribution on each cone is a linear function of the limit of the sequence $r_n (\bb_n - \bb)$.
To handle the fact that solutions to \eqref{primal} may not be unique, \cite{klatt2022limit} adopt a framework of an algorithmic flavor: they assume, informally speaking, that there exists a consistent, possibly randomized, selection procedure to specify a solution within the optimal set.
This strategy allows them to prove a distributional limit for the particular optimal solution selected by this procedure, without having to assume that the optimal solution is unique.

Despite the completeness and sophistication of their approach,  \cite{klatt2022limit} leave open several fundamental questions.
First, it is not clear whether it is possible to sample from their limit laws in polynomial time: all of their limits are expressed in terms of a decomposition of $\RR^m$ into a possibly exponential number of closed convex cones.
Even evaluating the functions involved  in their limiting expressions therefore appears to be computationally intractable.
Second, their approach to non-unique solutions cleverly sidesteps the need to assume that the optimal solution is unique; however, the resulting limit law does not give insight into the overall geometry of the random solution set.
Third, even ignoring issues of computational feasibility, their results do not yield a method to obtain asymptotically valid confidence sets from data, because the limiting distributions they obtain depend on the (typically unknown) optimal solutions to the original LP.

In this work, we propose solutions to these three questions.
First, in the case the solution to the original LP is unique, we give a new representation of the limit that can be sampled from in polynomial time; in fact, we show that the limit can be generated by solving an auxiliary random linear program.
Second, in the general (non-unique) case, we define and prove a distributional limit for the optimal solutions \emph{in the space of convex sets}---the resulting limit captures the random geometry of the entire solution set.
Finally, we develop a practical and computationally cheap data-driven method for constructing asymptotically valid confidence sets.

\section{Preliminaries on linear programming}\label{prelim}
In this section, we recall some facts about the structure of linear programs.
We point the reader to standard reference works \citep{NoceWrig06,bradley1977applied,boyd2004convex,BerTsi97} for additional background information.

We denote the set of optimal solutions to \eqref{primal} by
\begin{equation}
\displaystyle 
\bm{x}^*(\bm{b})\coloneqq\argmin_{\bm{x}\in\mathbb{R}^m} \langle\bm{c},\bm{x}\rangle,\qquad \textrm{s.t.}\ \bm{Ax}=\bm{b},\ \bm{x}\geq\bm{0}.
\label{solution_set}
\end{equation}
The notation $\bx^*(\bb)$ emphasizes that this optimal set depends on the right-hand side $\bb$.
In general, LPs do not possess unique solutions, so that typically $|\bm{x}^*(\bm{b})| \neq 1$.
However, if the solution is unique, by slight abuse of notation we write $\bm{x}^*(\bm{b})$ for both the (single-element) set of optimal solutions and for the optimal solution itself.
We sometimes refer to  $\bm{x}^*(\bm{b})$ as the set of ``target solutions,'' to contrast it with the random soultion set $\bx^*(\bb_n)$ obtained by replacing $\bb$ by its random counterpart. 
We denote the optimal objective value of \eqref{primal} by $\val(\bb)$.

Throughout, we make the following assumptions on \eqref{primal}.
\begin{assumption}
	The constraint matrix $\bm{A}$ has full rank, the optimal solution set $\bx^*(\bb)$ is nonempty and bounded, and~\eqref{primal} satisfies the Slater condition \citep[Section 5.2.3]{boyd2004convex}, i.e., $\exists \bm{x}_0\in\mathbb{R}^m$, such that $\bm{Ax}_0=\bm{b},\ \bm{x}_0>\bm{0}$.
\label{Assumption}
\end{assumption}
The assumption that $\bA$ is full rank is without loss of generality, as redundant constraints in the matrix can always be removed.
The assumption that $\bx^*(\bb)$ is nonempty and bounded is also made by \cite{klatt2022limit} and holds for many LPs of interest, including optimal transport problems.
Finally, the Slater condition is a standard assumption in convex programming and is only a minor strengthening of Assumption (B2) of \citet[see Lemma 5.4]{klatt2022limit}.

\subsection{Bases} For any subset $\I \subseteq \{1, \dots, m\}$, we denote by $\bm{A}_\I$ the $k \times |\I|$ submatrix of $\bm{A}$ formed by taking the columns of $\bm{A}$ corresponding to the elements of $\I$.
Analogously, for $\bm{x} \in \RR^m$, we write $\bm{x}_\I$ for the vector of length $|\I|$ consisting of the coordinates of $\bm{x}$ corresponding to $\I$.

\begin{definition}
	A set $\I \subseteq [m]$ is a \emph{basis} if
	\begin{equation}
	|\I|=k, \qquad \operatorname{rank}(\bm{A}_\I)=k\label{Intro:index_set}
	\end{equation}
\end{definition}
Given a basis $\I$, we can define the \emph{basic solution} $\bm{x}(\I;\bm{b})$ to be the vector $\bx$ satisfying 
\begin{equation}\label{basis_def}
	\begin{aligned}
	\bx_\I &= \bA_\I^{-1} \bb \\
	\bx_{\I^C} &= \bz\,.
	\end{aligned}
\end{equation}
Explicitly, $\bx(\I, \bb)$ is defined by setting the coordinates not in $\I$ to zero and inverting the matrix $\bA_\I$ to obtain the values on the coordinates corresponding to $\I$.
This vector is a feasible solution to~\eqref{primal} if and only if the vector $\bA_\I^{-1} \bb$ is nonegative; if it is, we say that $\bm{x}(\I;\bm{b})$ is a \emph{basic feasible solution}.
By construction, basic feasible solutions have at most $k$ non-zero entries: if we denote the support (i.e., the set of non-zero entries) of a vector $\bx$ by $S(\bx)$, then
\begin{equation*}
	S(\bx(\I; \bb)) \subseteq \I\,.
\end{equation*}
This inclusion can be strict if the vector $\bA_\I^{-1} \bb$ has zero coordinates.
When the inclusion is strict, the solution is called \emph{degenerate}.
If $\bx$ is a degenerate basic feasible solution, then any basis $\I$ such that $S(\bx) \subseteq \I$ satisfies $\bx = \bx(\I; \bb)$; in particular, several different bases may give rise to the same (degenerate) basic feasible solution.

Geometrically, basic feasible solutions are precisely extreme points (vertices) of the feasible set of \eqref{primal} \citep[Theorem 2.3]{BerTsi97}; we will therefore use the terms basic feasible solution and vertex interchangeably in what follows.
Our justification for focusing on basic feasible solutions is the ``fundamental theorem of linear programming''~\citep[Theorem 2.7]{BerTsi97}, which ensures that if any optimal solution to~\eqref{primal} exists, then there exists an optimum which is a basic feasible solution.

We denote by $\cI(\bb)$ the set of all bases $\I$ for which $\bx(\I; \bb)$ is a basic feasible solution, and by $\cI^*(\bb)$ the set of all bases $\I$ for which $\bx(\I; \bb)$ is an optimal solution.
The set of optimal vertices of \eqref{primal} is defined by
\begin{equation}
	\optvert(\bb) \defeq \{\bx(\I; \bb): \I \in \cI^*(\bb)\}\,.
	\label{optvert}
\end{equation}
The general theory of polyhedral geometry implies that since $\bx^*(\bb)$ is bounded, we may write $\bx^*(\bb) =  \operatorname{conv}(\optvert(\bb))$, the convex hull of $\optvert(\bb)$.
Moreover, the assumption that $\bx^*(\bb)$ is bounded implies that $\bx^*(\bb')$ is bounded for all perturbations $\bb'$.\footnote{This follows from the fact that $\bx^*(\bb)$ and $\bx^*(\bb')$ are polyhedra with the same recession cone, which must equal $\{\bz\}$ since $\bx^*(\bb)$ is bounded.}
We therefore also have $\bx^*(\bb') =  \operatorname{conv}(\optvert(\bb'))$.

\begin{table}
	\centering
	\begin{tabular}{@{}ll@{}} \toprule
		Symbol & Meaning\\ \midrule
		$\val(\bb)$ & Optimal objective value of \eqref{primal} \\
		$\bb$, $\bb_n$ & True and random right-hand side constraints, \eqref{b_limit}\\
		$S(\bx)$ & Set of nonzero coordinates of $\bx$ \\
		$\bx^*(\bb)$ & Set of optimal solutions, \eqref{solution_set} \\
		$\bx(\I; \bb)$ & Basic feasible solution, \eqref{basis_def} \\
		$\cI(\bb)$, $\cI^*(\bb)$& Set of feasible and optimal bases \\
		$\optvert(\bb)$ & Extreme points of $\bx^*(\bb)$, \eqref{optvert} \\
		$\supp_K$ & Support function of $K$, \eqref{support_funct} \\
		 \bottomrule
	\end{tabular}
	\caption{Important notation}
	\label{notation}
\end{table}

We summarize the main notation used in this paper in Table~\ref{notation}.

\section{Vertex and base stability}
This section presents two stability results that are central to our analysis.
Though simple and likely well known, we present them explicitly here to highlight the important role they play in our theorems.

The first is a Lipschitzian property of polytopes due to \cite{walkup1969lipschitzian}, which shows that the set of optimal solutions of \cref{primal} is Lipschitz with respect to the Hausdorff distance.
\begin{proposition}\label{lipschitz_vertex}
	Under \cref{Assumption}, there exists a constant $C = C(\bA, \bc) > 0$ such that if $\bb_1, \bb_2 \in \RR^k$ are such that $\bx^*(\bb_1)$ and $\bx^*(\bb_2)$ are nonempty, then $\haus(\bx^*(\bb_1), \bx^*(\bb_2)) \leq C \|\bb_1 - \bb_2\|$.
\end{proposition}

The second proposition shows that optimal bases for $\bb'$ are also optimal for $\bb$.
\begin{proposition}\label{no_disappearing}
	Under \cref{Assumption}, there exists $\delta = \delta(\bA, \bb) > 0$ such that if $\|\bb' - \bb\| \leq \delta$, then $\cI^*(\bb')$ is nonempty and $\cI^*(\bb') \subseteq \cI^*(\bb)$.
\end{proposition}

\section{A tractable limiting distribution when the target solution is unique}\label{unique}
In this section, we first consider the simplified setting where the target solution $\bx^*(\bb)$ is unique.
Even under this simplification, however, the limiting distribution obtained by \cite{klatt2022limit} does not have a tractable form.
In particular, it is not even clear whether it is possible to generate samples from this distribution in polynomial time.
The goal of this section is to obtain an expression for the limiting distribution that can be computed efficiently.

Stating this result requires defining a notion of distributional convergence suitable for a random set.
Even when $|\bx^*(\bb)| = 1$, it is possible that $|\bx^*(\bb_n)| > 1$.
This situation can arise when $\bx^*(\bb)$ is \emph{unique but degenerate}, i.e., when there exist multiple optimal bases in $\cI^*(\bb)$.
Even if these bases all give rise to the same solution $\bx^*(\bb)$ in the original program, they can correspond to \emph{different} optimal solutions when $\bb$ is replaced by $\bb_n$.
In this situation, $|\bx^*(\bb_n)| > 1$, and it is not possible to formulate a distributional limit for $r_n (\bx^*(\bb_n) - \bx^*(\bb))$ viewed as the difference of two vectors in $\RR^m$.
However, when $|\bx^*(\bb)| = 1$, we can consider the set defined by translating the elements of $\bx^*(\bb_n)$ by $\bx^*(\bb)$ and rescaling them by $r_n$:
\begin{equation*}
	r_n (\bx^*(\bb_n) - \bx^*(\bb)) \defeq \{r_n (\bx - \bx^*(\bb)): \bx \in \bx^*(\bb_n)\} \subseteq \RR^m\,.
\end{equation*}

Our first main result is that this random set enjoys a \emph{set-valued} distributional limit, with limit equal to the distribution of the optimal set of a random auxiliary linear program.\footnote{To define weak convergence in this setting, we view these random sets as random elements in the metric space of compact subsets of $\RR^m$ equipped with the Hausdorff distance, and weak convergence means, as usual, the convergence of expectations of bounded, continuous functions in this topology \citep{molchanov2005theory,Kin89}.}
\begin{theorem}\label{thm:unique}
	Suppose that \cref{primal} satisfies \cref{Assumption}.
	If $\bb_n$ satisfies the distributional limit \cref{b_limit} and $|\bx^*(\bb)| = 1$, then
	\begin{equation}\label{unique_convergence}
		r_n (\bx^*(\bb_n) - \bx^*(\bb)) \weakto \bp^*_\bb(\GG)\,,
	\end{equation}
	where $\bp^*_\bb(\GG)$ is the set of optimal solutions to the following linear program:
	\begin{equation}
		\min \langle \bc, \bp \rangle: \bA \bp = \GG, \quad \bp_{i} \geq 0 \quad \forall i \notin S(\bx^*(\bb))\,.
		\label{unique_auxiliary_lp}
	\end{equation}
\end{theorem}

The continuous mapping theorem implies that continuous functionals of the set $r_n(\bx^*(\bb_n) - \bx^*(\bb))$ also enjoy weak convergence.
To give a concrete example of the statistical implications of this fact, consider the problem of obtaining a confidence set for $\bx^*(\bb)$.
Doing so requires knowing how far $\bx^*(\bb)$ typically is from $\bx^*(\bb_n)$.
If we let $d(S, \bx) = \inf_{\bm{y} \in S} \|\bm{y} - \bm{x}\|$, then the following corollary shows that we can obtain a distributional limit for $d(\bx^*(\bb_n), \bx^*(\bb))$.
\begin{corollary}\label{unique_cor}
$r_n d(\bx^*(\bb_n), \bx^*(\bb)) \weakto d(\bp^*_\bb(\GG), \bz)$.
\end{corollary}

In words, the rescaled distance of the target solution $\bx^*(\bb)$ to the set of optimal solutions of the random program converges in distribution to the distance of zero to the optimal set of the random auxiliary LP. Importantly, this is a convex program, whose solution can be found in polynomial time.

Let us compare Corollary~\ref{unique_cor} with what would be obtained by a more standard approach.
If one finds estimators by solving an optimization problem that can yield multiple optima, a standard path to inference consists in first 
identifying  a subset of them that  are close to one another, and then deriving the limiting distribution of any one of them, relative to the unique target.  
By contrast, Corollary \ref{unique_cor} gives information about the distance of $\bx^*(\bb)$ to the whole set of optima for the random program.

We stress our limit law is equivalent to the one obtained by \citet[Theorem 3.5]{klatt2022limit}.
The benefit of \cref{thm:unique} is that $\bp^*_\bb(\GG)$ is given explicitly: though this set can be large, it is algorithmically accessible since it possesses an explicit polyhedral representation in terms of separating hyperplanes. This implies, for instance, that it is possible to solve convex optimization problems involving $\bp^*_\bb(\GG)$ in polynomial time via the ellipsoid method~\cite{}. On the other hand, \citet{klatt2022limit} prove the same result but where the expression on the right side is a sum over a decomposition of $\RR^m$ into a possibly exponential number of pieces; such a decomposition typically cannot be evaluated in polynomial time.

When the unique optimal solution $\bx^*(\bb)$ is also non-degenerate, then \cref{no_disappearing} implies that for $\bb_n$ sufficiently close to $\bb$, the perturbed linear program also possesses a unique solution.
In this situation, \cref{thm:unique} is a \emph{standard} distributional limit: asymptotically almost surely, the set $\bx^*(\bb_n)$ reduces to a singleton, and \cref{thm:unique} shows that the distributional limit of the vector $r_n(\bx^*(\bb_n) - \bx^*(\bb))$ is the (unique) solution to \cref{unique_auxiliary_lp}, which is just $\bx(\I^*; \GG)$ for the unique $\I^* \in \cI^*(\bb)$.
This recovers the limit for this simplified setting mentioned by \citet[discussion after Remark 3.2]{klatt2022limit}.

\section{Distributional convergence in the space of convex sets}\label{non-unique}
When $\bx^*(\bb)$ is not unique, the approach to defining a set-valued distributional limit taken in \cref{thm:unique} no longer succeeds.
Indeed, if $\bx^*(\bb_n)$ and $\bx^*(\bb)$ are general closed sets, then even if $\bx^*(\bb_n) \to \bx^*(\bb)$ in Hausdorff distance, the set
\begin{equation*}
	\bx^*(\bb_n) \ominus \bx^*(\bb) \defeq \{\bx - \bx': \bx \in \bx^*(\bb_n), \bx' \in \bx^*(\bb)\}
\end{equation*}
will not converge to $\{\bz\}$ in general, so that no meaningful limit of $r_n (\bx^*(\bb_n) \ominus \bx^*(\bb))$ exists.
In the non-unique case, \cite{klatt2022limit} therefore define a distributional limit under the additional assumption that there exists a consistent scheme for selecting a single element of $\bx^*(\bb_n)$ and $\bx^*(\bb)$; they then show that this selection satisfies a distributional limit in the classical sense.
This ingenious approach captures the behavior of practical algorithms for solving LPs, since reasonable LP solvers give rise to such selection schemes \cite[see][Lemma 5.5]{klatt2022limit}.
However, as in the case where the target solution is unique, their limiting distribution is expressed as a sum over a decomposition of $\RR^m$ into a possibly exponential number of pieces.
Moreover, their techniques do not give insight into the overall fluctuations of the random set $\bx^*(\bb_n)$.
By contrast, in the unique case, Theorem~\ref{thm:unique} shows that it is possible to obtain simultaneous control over the whole random set.

In this section, we leverage techniques from random convex geometry to obtain similar results for the non-unique case.
Unlike Theorem~\ref{thm:unique}, Theorem~\ref{thm:non-unique} goes beyond the setting analyzed by \cite{klatt2022limit}.
Like Theorem~\ref{thm:unique}, we state our convergence results in terms of the optimal solutions to a random auxiliary LP, implying that evaluating the limits we obtain can be computationally tractable in applications.

To formulate our distributional limit, we adopt a strategy developed by \cite{artstein1975strong}, \cite{weil1982application}, and independently by \cite{lyashenko1983statistics} to prove central limit theorems for random compact sets.
To any compact, convex set $K \subseteq \RR^m$, we associate its \emph{support function} $\supp_K: \mathbb{S}^{m-1} \to \RR$ defined by
\begin{equation}
	\supp_K(\alpha) \defeq \sup_{\bx \in K} \langle \alpha, \bx \rangle\,.
	\label{support_funct}
\end{equation}
The mapping $K \mapsto \supp_K$ provides an \emph{isometric embedding} of the metric space of convex, compact sets equipped with the Hausdorff metric into the Banach space $\cC(\mathbb{S}^{m-1})$ of continuous functions on the sphere equipped with the uniform norm \cite[see][section 3.1.2]{molchanov2005theory}.
Explicitly, given two compact, convex sets $K_1$ and $K_2$, we have
\begin{equation}\label{isometry}
	\haus(K_1, K_2) = \sup_{\alpha \in \mathbb{S}^{m-1}} |\supp_{K_1}(\alpha) - \supp_{K_2}(\alpha)|\,.
\end{equation}
In particular, the map from a convex set to its support function is injective; $K$ can be recovered from $\supp_K$ by taking its Legendre transform.
This embedding has two profound implications.
First, the geometry of convex sets is entirely captured by their support functions.
In particular, we may associate to a random convex set its support function, viewed as a random element of $\cC(\mathbb{S}^{m-1})$, and study its distribution instead.\footnote{We omit a detailed discussion of measurability here, but it can be shown that if the space of convex, compact subsets of $\RR^m$ is equipped with an appropriate $\sigma$-algebra (known as the Effros $\sigma$-algebra), then for a random set $K$ the support function $\supp_K$ is indeed a random variable in $\cC(\mathbb{S}^{m-1})$ \cite[see][Proposition 2.5]{molchanov2005theory}.}
Second, since $\cC(\mathbb{S}^{m-1})$ is a Banach space, we may leverage the theory of probability in Banach spaces to prove limit theorems for support functions.

Our main result of this section is a distributional limit for the set $\bx^*(\bb_n)$.
Once again, it is stated in terms of the solutions to an auxiliary linear program.
\begin{theorem}\label{thm:non-unique}
	Let $\supp_n$ and $\supp$ be the support functions of $\bx^*(\bb_n)$ and $\bx^*(\bb)$, respectively.
	Suppose that \eqref{primal} satisfies \cref{Assumption}.
	If $\bb_n$ satisfies the distributional limit \eqref{b_limit}, then
	\begin{equation}
		r_n(\supp_n - \supp) \weakto \limfunc{\GG}\,,
	\end{equation}
	where $\limfunc{\GG}$ is the random element of $\cC(\mathbb{S}^{m-1})$ defined by
	\begin{equation*}
		\limfunc{\GG}(\alpha) = \sup_{\bx \in \mathbf{q}^*_{\alpha}(\GG)} \langle \alpha, \bx \rangle\,,
	\end{equation*}
	and $\mathbf{q}^*_{\alpha}(\GG)$ is the set of optimal vertex solutions to the following linear program:\footnote{The function $\supp$ is only differentiable almost everywhere, but since $\limfunc{\GG}(\alpha)$ is almost surely continuous it suffices to specify its values on a dense subset.}
	\begin{equation}\label{non_unique_lp}
		\min \langle \bc, \bm{q} \rangle: \bA \bm{q} = \GG, \quad \bm{q}_{i} \geq 0 \quad \forall i \notin S(\nabla \supp(\alpha))\,.
	\end{equation}
\end{theorem}
Informally, \cref{thm:non-unique} shows that when $n$ is large, $\supp_n \overset{d}{\approx} \supp + r_n^{-1}\limfunc{\GG}$.
By the isometry described in \cref{isometry}, this translates into a statement about the fluctuations of the random set $\bx^*(\bb_n)$ around $\bx^*(\bb)$.
The proof of Theorem~\ref{thm:non-unique} is based on establishing the directional Hadamard differentiability of the mapping $\bb \mapsto \supp_{\bx^*(\bb)}$ viewed as a function from $\RR^k$ to $\cC(\mathbb{S}^{m-1})$, and then applying a functional delta method due to \cite{romisch2004delta}.

Like Theorem~\ref{thm:unique}, Theorem~\ref{thm:non-unique} has statistical implications for the problem of obtaining a confidence set for $\bx^*(\bb)$.
The isometry \eqref{isometry} implies the following analogue of Corollary~\ref{unique_cor}.
\begin{corollary}\label{non_unique_cor}$r_n \haus(\bx^*(\bb_n), \bx^*(\bb)) \weakto \sup_{\alpha \in \mathbb{S}^{m-1}} |\limfunc{\GG}(\alpha)|$
\end{corollary}
In other words, the rescaled Hausdorff distance between the solution sets converges in distribution to  the supremum of a random continuous function on the sphere.
Corollary~\ref{non_unique_cor} can be compared to \cite[Proposition 3.7]{klatt2022limit}, which shows that $\haus(\bx^*(\bb_n), \bx^*(\bb)) = O_P(r_n^{-1})$.
Our result gives finer control over the behavior of the rescaled distance in terms of the solutions to auxiliary linear programs.
However, unlike Corollary~\ref{unique_cor}, we are not aware of an algorithm that can compute the supremum on the right side of Corollary~\ref{non_unique_cor} in polynomial time.
Finding a computationally tractable expression for this limit is an attractive open problem.

\section{Data-driven confidence sets}\label{confidence}
\Cref{thm:unique,thm:non-unique} give explicit distributional limits for $\bx^*(\bb_n)$ in terms of auxiliary linear programs.
Though evaluating these limits is computationally tractable, they fail to be suitable for concrete inference tasks because the limiting distributions depend on properties of the true optimal solution set $\bx^*(\bb)$.
Since this set is almost always unknown in practice, Theorems~\ref{thm:unique} and~\ref{thm:non-unique} do not provide a data-driven way to obtain asymptotically valid confidence sets.

In principle, the fact that Theorem~\ref{thm:non-unique} is proven by directional Hadamard differentiability arguments implies that the $m$-out-of-$n$ bootstrap is consistent \citep{Dum93}.
However, using the bootstrap for inference raises other practical difficulties: it is an open question how to choose $m$ for good performance, and convergence is slow.
Therefore, even though Theorems~\ref{thm:unique} and~\ref{thm:non-unique} provide a complete answer to the theoretical question of obtaining a valid distributional limit, they are a poor way to construct confidence sets in practice.

In this section, we give a simple procedure to obtain such sets.
Specifically, we suppose that that statistician has solved the perturbed linear program and obtained a random solution $\hat \bx_n \in \optvert(\bb_n)$ along with a corresponding basis $\I_n \in \cI^*(\bb_n)$.\footnote{Algorithms such as the simplex method always return an optimal vertex when one exists, along with a corresponding basis \cite[Theorem 3.3]{BerTsi97}.}
We will construct a confidence set based on $\hat \bx_n$ that is guaranteed to contain at least one element of $\bx^*(\bb)$ with high probability.
Specifically, let us consider the basic solution $\bx(\I_n; \bb)$ defined by the random basis $\I_n$.
This solution may not be feasible for \eqref{primal}, much less optimal, so we define the projection
\begin{equation}\label{proj}
	\proj \defeq  \argmin_{\bm{x}\in \bm{x^*}(\bm{b})} \|\bm{x}(\I_n;\bm{b})-\bm{x}\|\,.
\end{equation}

The following result shows that we can construct a set containing this point with high probability.

\begin{theorem}\label{confidence_set}
	Suppose that \eqref{primal} satisfies \cref{Assumption} and $\bb_n$ satisfies the distributional limit \eqref{b_limit}. 
	Let $G_\alpha$ be an open set such that $\p{\GG \in G_\alpha} \geq 1 - \alpha$.
	Then
	\begin{equation}
		\liminf_{n \to \infty} \mathbb{P}\left(r_n(\hat{\bm{x}}_n - \proj)\in \bm{x}(\I_n;G_\alpha)\right) \geq 1-\alpha\,,
		\label{Confidence set}
	\end{equation}
	where $\bm{x}(\I_n;G_\alpha) \defeq \{\bm{x}(\I_n;\bm{G}): \bm{G} \in G_\alpha\}$.
\end{theorem}
\begin{corollary}[Confidence set for an optimal solution]\label{confidence_cor}
	In the setting of Theorem~\ref{confidence_set}, the set $C_n \defeq \{\hat \bx_n - r_n^{-1} \bx: \bx \in \bm{x}(\I_n;G_\alpha)\}$ contains an element of $\bx^*(\bb)$ with asymptotic probability at least $1-\alpha$.
\end{corollary}
\Cref{confidence_set} and Corollary~\ref{confidence_cor} are weaker than \cref{thm:non-unique,thm:unique}: they do not give any information about the whole set of optimal solutions $\bx^*(\bb)$.
Instead, Corollary~\ref{confidence_cor} only guarantees that $C_n$ contains \emph{an} optimal solution with high probability.
As our simulations in \cref{sec:examples} show, when the optimal solution is non-unique, the confidence sets constructed by this procedure sometimes cover one solution, sometimes another.
Nevertheless, \cref{confidence_cor} does offer the practitioner an asymptotic guarantee that \emph{some} optimal solution is in a neighborhood of the estimator.

On the other hand, unlike \cref{thm:non-unique,thm:unique}, Corollary~\ref{confidence_cor} is eminently practical.
It requires only the outputs $\hat \bx_n$ and $\I_n$ from a standard linear programming algorithm, and the set $\bx(\I_n: G_\alpha)$ is easy to compute, since the mapping $\bm{G} \mapsto \bx(\I_n: \bm{G})$ is an explicit linear transformation.
For instance, if $G_\alpha$ is an ellipsoid of the form $\{\bm{y} \in \RR^k: \bm{y}^\top \Sigma^{-1} \bm{y} < 1\}$, then recalling definition (\ref{basis_def}) in \cref{prelim} we have 
\begin{equation}
	\bx(\I_n; G_\alpha) = \{\bx \in \RR^m: \bx_{\I_n}^\top \bm{M}_n \bx_{\I_n}^\top < 1,\, \bx_{\I_n^c} = \bz\}\,,
\end{equation}
where $\bm{M}_n \defeq \bA_{\I_n}^\top \Sigma^{-1} \bA_{\I_n} \in \RR^{k \times k}$.

\section{Examples}\label{sec:examples}
We will provide two examples in this section to show the effectiveness of the method described in \cref{confidence_set} and Corollary~\ref{confidence_cor} for generating a confidence set for solutions to LPs.

We first return to the simple discrete optimal transport problem described in the introduction, which is a linear program with a unique degenerate optimal vertex.
We then treat a more complicated example arising from a min-cost flow problem \cite[see][section 8.1]{bradley1977applied}.
In this example, there are two optimal vertex solutions at the population level.
In both cases, our simulations confirm that the method gives confidence sets which cover an optimal solution with high probability.
\subsection{Empirical Optimal Transport}
We consider the optimal transport example given in the introduction, where we suppose that $n \bm{r}_n \sim \mathrm{Mult}(n, (1/2, 1/2))$.
This corresponds to the situation where we aim to estimate the solution to an optimal transport problem involving an unknown distribution $\bm{r} = (1/2, 1/2)$ on the basis of $n$ i.i.d.\ samples from $\bm{r}$.
In this setting, the classical central limit theorem implies $\sqrt n(\bm{r}_n - \bm{r}) \weakto (Z, -Z)$, where $Z \sim \cN(0, 1/4)$.
We therefore choose $G_\alpha = \{(x, -x): x \in [-z_{0.025}/2, z_{0.025}/2]\}$, where $[-z_{0.025}/2, z_{0.025}/2]$ is a 95\% confidence interval for an $\cN(0, 1/4)$ random variable, and use Corollary~\ref{confidence_cor} to construct a confidence set for the entries of $\pi$.

\Cref{fig:OT2X2} shows examples of the confidence intervals produced by our method.
We plot one realization for each of the labeled values of $n$.
Note that for each realization, the confidence intervals for two (random) entries of $\pi$ are singletons: for example, when $n=20$, the solution we obtained to the LP was $\pi_n=(0.5,0.05;0,0.45)$ and the confidence intervals given by Corollary~\ref{confidence_cor} were $\pi_{11}=0.5$, $\pi_{12}\in[-0.169,0.269]$, $\pi_{21}=0$ and $\pi_{22}\in[0.23,0.67]$.
Even though the confidence intervals for $\pi_{11}$  and $\pi_{21}$ have zero width, this set does in fact contain the optimal solution $(\frac{1}{2},0;0,\frac{1}{2})$.
The somewhat counterintuitive fact that a confidence set with empty interior covers the true parameter with probability approaching $95\%$ is a consequence of the fact that the distribution of $\hat \pi_n$ is \emph{not} absolutely continuous with respect to the Lebesgue measure.

\begin{figure}
    \centering
    \includegraphics[width=\textwidth]{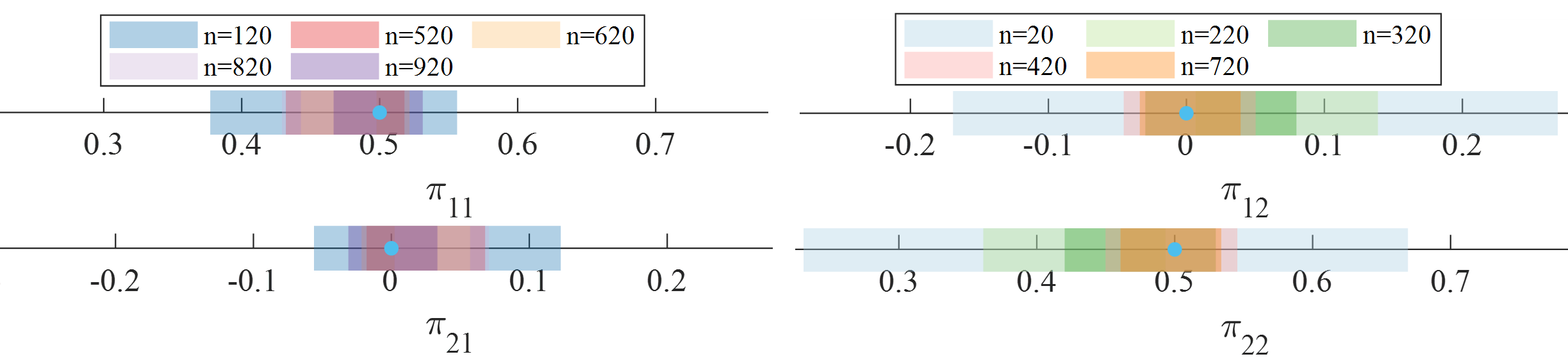}
    \caption{Example confidence intervals for $\pi=(\frac{1}{2},0;0,\frac{1}{2})$ computed with different values of $n$ (one replicate each). For the values of $n$ appearing in the box on the left, the confidence intervals for $\pi_{12}$ and $\pi_{22}$ were singletons at $0$ and $1/2$, respectively; for the values of $n$ appearing on the right, the confidence intervals for $\pi_{11}$ and $\pi_{12}$ were singletons.}
    \label{fig:OT2X2}
\end{figure}

We also estimate the observed coverage probabilities for finite $n$.
For each $n$, we generate $1000$ independent replicates, calculate the $95\%$ confidence intervals and count the replicates that successfully capture a true solution.
\begin{center}
\begin{tabular}{ |c|c|c|c|c|c|c|c|c| } 
 \hline
   n & $1$ & $3$ & $5$ & $10$ & $50$ & $100$ & $500$ & $10000$ \\ 
 \hline
 Coverage Probability &  $0.480$& $0.892$& $0.941$&$0.981$ &$0.935$&$0.922$&$0.947$&$0.950$\\ 
 \hline

\end{tabular}
\end{center}

\subsection{Minimal Cost Flow Problem}
We adapt an example from  \citet[Section 8.1]{bradley1977applied} arising in operations research.
Consider the problem of moving goods from origins to destinations along routes with certain volume constraints and costs.
We model an instance of this problem as the directed graph depicted in \cref{fig: min cost flow network}, with $5$ nodes and $9$ arcs. Each arc is unidirectional, labeled with its capacity and transportation cost (the pair of numbers in the parentheses adjacent to the arc). Each node is labeled with its supply or demand. For example, the supply of node $1$ is  $20$. The arc $x_{12}$ transports products from node $1$ to node $2$ with the maximum capacity of $15$ units of product and the cost $\$4$ per unit of product.
Assuming that the total demand matches the total supply, the goal is to fulfill all the demands in the network at a minimum cost.

\begin{figure}
    \centering
    \includegraphics[width=0.5\textwidth]{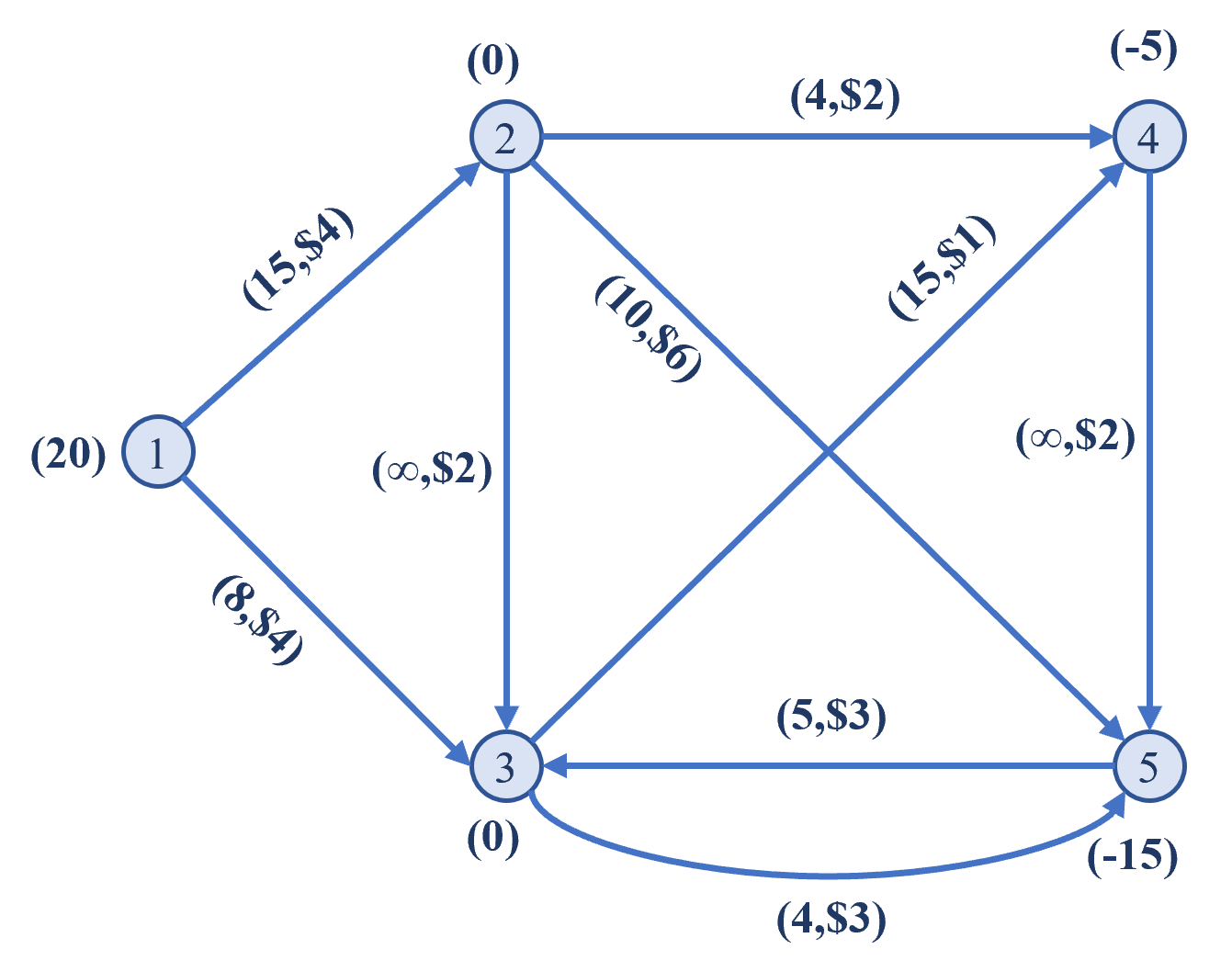}
    \caption{Minimal-cost flow problem. Each arc is labeled with its \emph{capacity} (the total amount of flow it can carry) and the cost of moving a single unit of flow across it. Vertices are labeled with supplies (positive quantities) or demands (negative quantities) for goods at each location.}
    \label{fig: min cost flow network}
\end{figure}

This minimal-cost flow problem can be written in a linear program form:
\begin{equation}
    \label{Example LP}
    \min \sum_{i,j}c_{ij}x_{ij} :\sum_j x_{ij}-\sum_k x_{ki}=b_i\ (i=1,2,...,5),\ 
    0\leq x_{ij}\leq u_{ij},
\end{equation}
where $b_i$ is the supply of each node, $u_{ij}$ is the capacity of each arc, and $ c_{ij}$ is the transportation cost of each arc. A standard linear program in the form of \cref{primal} can be obtained for this problem by introducing the auxiliary variable $y_{ij}$, which satisfies $y_{ij}+x_{ij}=u_{ij}$ and $y_{ij}\geq 0$. The auxiliary variable $y_{ij}$ represents the remaining capacity for each arc. The standard form for \cref{Example LP} is
\begin{equation}
     \min \sum_{i,j}c_{ij}x_{ij} :\sum_j x_{ij}-\sum_k x_{ki}=b_i, \quad  y_{ij}+x_{ij}=u_{ij},\ x_{ij}\geq 0,\  y_{ij}\geq 0.
\end{equation}
Note that the equality constraints $\sum_j x_{ij}-\sum_k x_{ki}=b_i$ are redundant due to the flow balance condition of the network, and deleting any one of them will not change the program. Suppose the flow balance constraint on the third node is deleted and we have the modified supply vector $\tilde{\bb}=(b_1,b_2,b_4,b_5)=(20,0,-5,-15)$.

The program in \cref{fig: min cost flow network} has two optimal vertex solutions: 
\begin{center}
\begin{tabular}{ |c|c|c|c|c|c|c|c|c|c| } 
 \hline
  & $x_{12}$ & $x_{13}$ & $x_{23}$ & $x_{24}$ & $x_{25}$ & $x_{34}$ & $x_{35}$ & $x_{45}$ & $x_{53}$ \\ 
 \hline
 solution 1 &  12& 8& 8&4 &0&15&1&14&0\\ 
 \hline
 solution 2 & 12& 8& 8&4 &0&12&4&11&0 \\ 
 \hline
\end{tabular}
\end{center}

In applications, the true supply and demand at each node may not be known precisely, but rather must be estimated by an empirical supply vector $\tilde{\bb}_n$ obtained by averaging the observed supplies and demands over $n$ days.
Suppose that we know $\sqrt{n}(\tilde{\bb}_n-\tilde{\bb})\weakto \mathbb{G}_0$, where $\mathbb{G}_0\sim\mathcal{N}\left(0,\mathrm{diag}(4,1,1,3)\right)$. 
We calculate a min-cost flow $\hat{\bx}_n$ using the estimated supply vector $\tilde{\bb}_n$, and employ Corollary~\ref{confidence_cor} to build a confidence set.

To visualize the confidence set for $\hat{\bx}_n$ for various $n$, we show the projection of $4$ dimensional confidence sets to lower dimensional spaces. As an example, we plot the confidence interval for the $x_{45}$ coordinate (\cref{fig:1Dmultiple}, \cref{fig:1D n50}) and the confidence set for the $2$ dimensional arc pair $(x_{23},x_{45})$ (\cref{fig:2D n50}).
In \cref{fig:1Dmultiple}, we show several examples of the confidence sets we obtain for $x_{45}$.
We plot a single realization for each value of $n$.
\Cref{fig:1D n50} shows many replicates for the $n = 50$ case to illustrate the sampling variability of the sets we construct, and \cref{fig:2D n50} depicts the same procedure for the two-dimensional confidence set for $(x_{23}, x_{45})$.
We can see that for each replicate, the given confidence sets capture \emph{one} of the solutions very well---which solution is covered depends on the random fluctuations in each replicate.

In short, Corollary~\ref{confidence_cor} gives a practical means of obtaining asymptotically valid confidence sets for the solution to a linear program.
To our knowledge, this is the first procedure satisfying these requirements.

\section*{Acknowledgements}
Bunea was supported in part by NSF grant DMS-2210563, and Niles-Weed was supported in part by NSF grant DMS-2210583 and a Sloan Research Fellowship. 
\begin{figure}
    \centering
    \includegraphics[width=\textwidth]{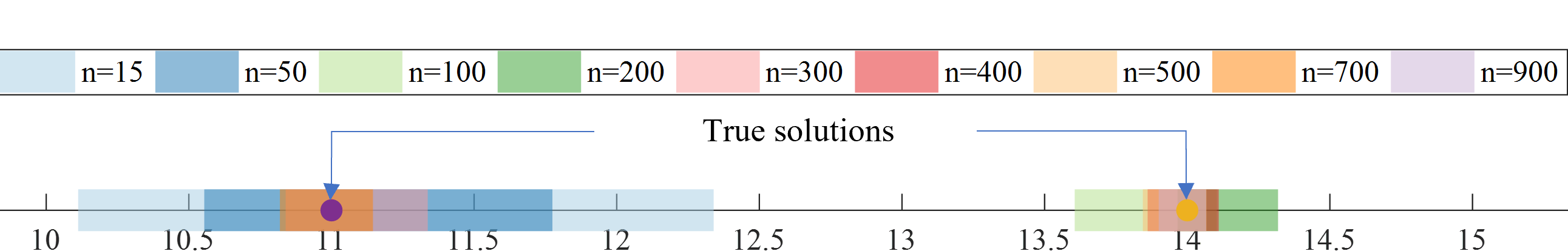}
    \caption{Example confidence intervals for flow through arc $x_{45}$ computed with different values of $n$.}
    \label{fig:1Dmultiple}
\end{figure}
\begin{figure}
    \centering
    \includegraphics[width=\textwidth]{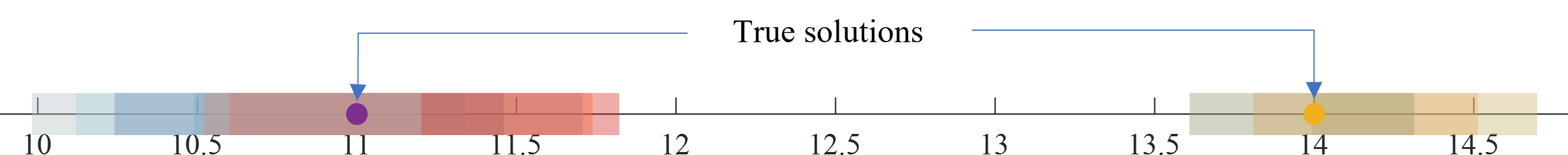}
    \caption{Confidence intervals for flow through arc $x_{45}$ when $n=50$ (many replicates).}
    \label{fig:1D n50}
\end{figure}
\begin{figure}
    \centering
    \includegraphics[width=\textwidth]{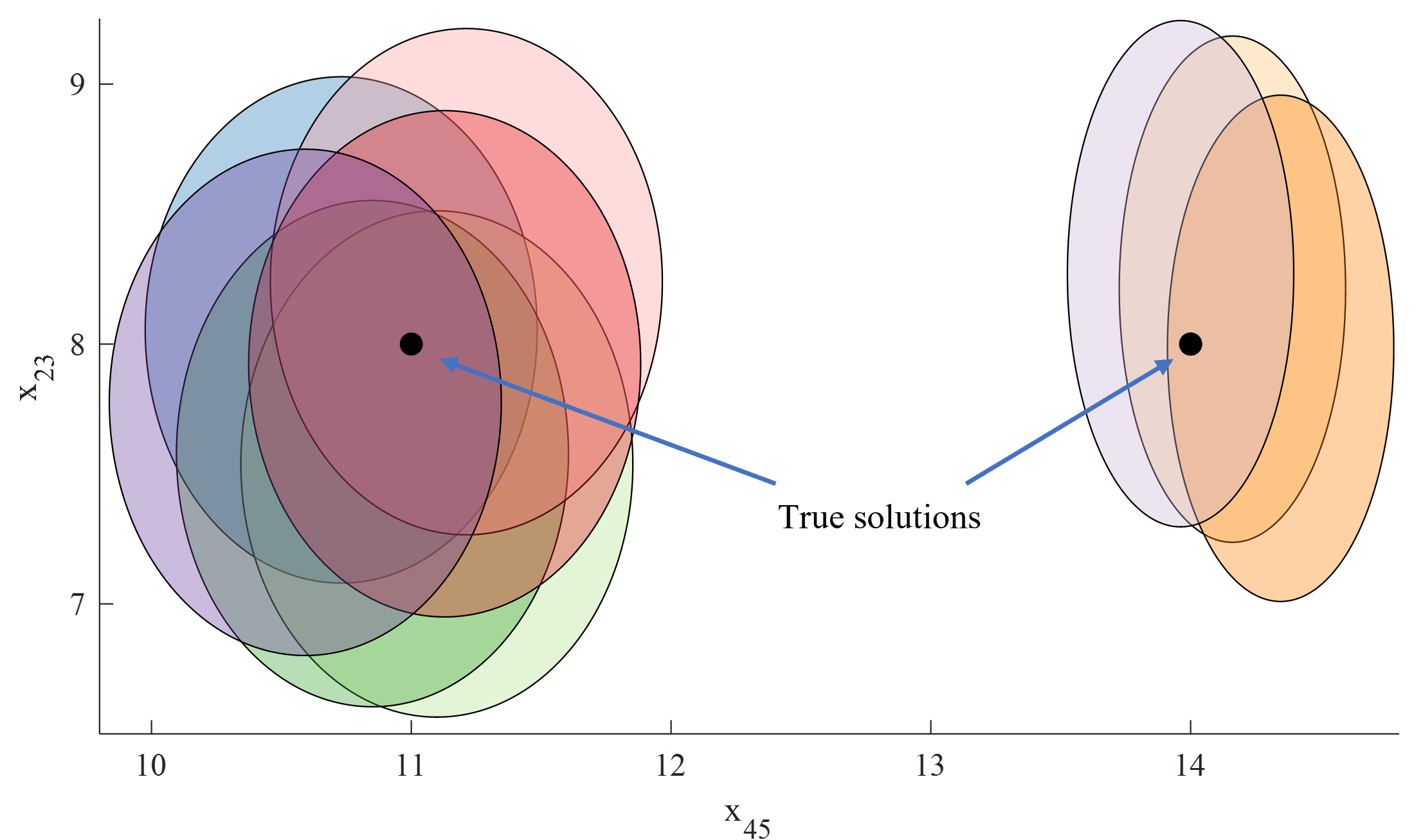}
    \caption{Confidence sets for flow through the arc pair $(x_{23},x_{45})$ when $n=50$ (many replicates).}
    \label{fig:2D n50}
\end{figure}

\appendix
 \crefalias{section}{appendix} \section{Proofs of propositions}
We first establish a few basic lemmas. In the proofs, we utilize the optimal conditions of the linear program \cref{primal} and its dual program:
\begin{equation}
    \max_{\bm{\lambda}\in \mathbb{R}^k} \langle \bm{b}, \bm{\lambda}\rangle,\qquad \textrm{s.t.}\ \bm{c}-\bm{A^T\lambda}\geq 0. 
    \label{Intro:dual}
\end{equation}
The linear program \cref{primal} and the dual program \cref{Intro:dual} achieve their optima at $(\bx^*(\bb),\ \bm{\lambda}^*(\bb))$ if and only if $\exists\  \bm{s}\in\mathbb{R}^m$ such that:
\begin{equation}
    \label{optimality condition}
    \bA^T\bm{\lambda}^*(\bb)+\bm{s}=\bc,\ \bA\bx^*(\bb)=\bb,\ \bx^*(\bb)\geq 0,\ \bm{s}\geq 0, \ \bx^*(\bb)^T\bm{s}=0. 
\end{equation}
The last condition is called \emph{complementary slackness}, and is equivalent to the condition that $ \bx^*(\bb)_i > 0 \implies \bm{s}_i = 0$ for all $i \in [m]$.
\begin{lemma}\label{continuity_lemmas}
	Under \cref{Assumption}, there exists $\delta = \delta(\bA, \bb) > 0$, $C_1 = C(\bA)$, and $C_2 = C(\bA, \bc)$  such that the following properties hold:
	\begin{enumerate}
		\item If $\|\bb' - \bb\| \leq \delta$, then $\cI(\bb') \subseteq \cI(\bb)$
		\item If $\|\bb' - \bb\| \leq \delta$, then $\bx^*(\bb') \neq \emptyset$
\item $\|\bx(I; \bb') - \bx(I; \bb)\| \leq  C_1 \|\bb' - \bb\|$, \ for all \ $I \in \cI(\bb')$.
		\item If $\val(\bb')$ is finite, then $|\val(\bb') - \val(\bb)| \leq C_2 \|\bb' - \bb\|$,
	\end{enumerate} 
\end{lemma}
\begin{proof}
The perturbed LP with linear constraint $\bA\bx=\bb'$ reads:
 \begin{equation}
	\min_{\bm{x}\in\mathbb{R}^m} \langle\bm{c},\bm{x}\rangle,\qquad \text{s.t.}\ \bm{Ax}=\bm{b'},\ \bm{x}\geq\bm{0},
 \label{Intro:total_p}
\end{equation}
\paragraph{1. Inclusion of feasible bases: $\cI(\bb') \subseteq \cI(\bb)$}
If $\exists \I_0\in \mathcal{I}(\bm{b}')\backslash\mathcal{I}(\bm{b})$, there exist $1\leq p\leq k$ such that $ (\bm{A}_{\I_0}^{-1}\bm{b}')_p\geq 0$ and  $ (\bm{A}_{\I_0}^{-1}\bm{b})_p<0$. 
However, when $\|\bm{b}'-\bm{b}\|<\frac{|(\bm{A}_{I_0}^{-1}\bm{b})_p|}{\|\bA^{-1}_{I_0}\|}$, $$|(\bm{A}_{I_0}^{-1}\bm{b}')_p-(\bm{A}_{I_0}^{-1}\bm{b})_p|\leq\|\bm{A}_{I_0}^{-1}\bm{b}'-\bm{A}_{I_0}^{-1}\bm{b}\|<|(\bm{A}_{I_0}^{-1}\bm{b})_p|.$$ Therefore, $(\bm{A}_{I_0}^{-1}\bm{b'})_p<0$. Take $\delta_{b_0}=\min_{\{I_0|\text{$\bA_{I_0}$ invertible}\}} \min_{p: (\bA_{I_0}^{-1}\bb)_p < 0} \frac{|(\bm{A}_{I_0}^{-1}\bm{b})_p|}{\|\bA_{I_0}^{-1}\|}$. When $\|\bm{b}'-\bm{b}\|<\delta_{b_0}$, there is no such basis $\I_0$ and  $\mathcal{I}(\bm{b}')\subseteq\mathcal{I}(\bm{b})$. 

\paragraph{2. Existence of optimal solution $\bx^*(\bb')$.}
We first show that the perturbed LP is feasible.

By \cref{Assumption}, there exists $\bx_0$ satisfying
\begin{equation*}
	\bm{A}\bm{x}_0=\bm{b},\qquad \bm{x}_0>0\,.
\end{equation*}
Let $s_{\bm{x}_0}$ be the smallest entry of $\bm{x}_0$ and let $I_0$ be an arbitrary element of $\mathcal{I}(\bm{b})$.
When $\|\bm{b}'-\bm{b}\|<\frac{s_{\bm{x}_0}}{\|\bm{A}_{I_0}^{-1}\|}\coloneqq \delta_{b_1}$, we have $$\|\bm{x}(I_0;\bm{b}'-\bm{b})\|=\|\bm{A}_{I_0}^{-1}(\bm{b}'-\bm{b})\|<s_{\bm{x}_0}.$$
Then $\bm{x}_0'\coloneqq\bm{x}_0+\bm{x}(I_0;\bm{b}'-\bm{b})$ satisfies $\bm{A}\bm{x}_0'=\bm{b}'$ and $\bm{x}_0'>0$, which indicates that $\bm{x}_0'$ lies in the feasible region of \cref{Intro:total_p}.

The dual problem of \cref{Intro:total_p} is
\begin{equation}
	\max_{\bm{\lambda}\in \mathbb{R}^k} \langle \bm{b}', \bm{\lambda}\rangle,\qquad \textrm{s.t.}\ \bm{c}-\bm{A^T\lambda}\geq 0. 
	\label{Intro:dual_p}
\end{equation}
The fact that $\bx^*(\bb)$ is nonempty implies that \cref{Intro:dual_p} is feasible, since the feasible set of \cref{Intro:dual_p} does not depend on $\bb'$.
Hence the value of \cref{Intro:total_p} is bounded and there exist optimal solutions.

\paragraph{3. Lipschitz continuity of basic feasible solutions.}
We argue as in Part 1.
For any $\I \in \cI(\bb')$, $\bx(\I; \bb')_{\I^C} = \bx(\I; \bb)_{\I^C} = \bz$, and
\begin{equation*}
	\|(\bm{A}_{\I}^{-1}\bm{b}')-(\bm{A}_{\I}^{-1}\bm{b})\|=\|\bm{A}_{\I}^{-1}\bm{b}'-\bm{A}_{\I}^{-1}\bm{b}\| \leq \|\bA_\I^{-1}\| \|\bb' - \bb\|\,.
\end{equation*}
Taking $C_1 = \max_{\I: \bA_\I \text{ invertible}} \|\bA_\I^{-1}\|$ yields the bound.

\paragraph{4. Local Lipschitz continuity of optimal value}
The fact that target and perturbed primal problems have finite values indicates that there exist optimal solutions to the target and perturbed dual problems.
Denote optimal vertex solutions to the dual programs by  $\bm{\lambda}^*(\bm{b})$ and $\bm{\lambda}^*(\bm{b}')$, respectively
Strong duality implies
$$\langle \bm{b}', \bm{\lambda}^*(\bm{b}') \rangle = f(\bm{b}')= \langle \bm{c}, \bm{x}^*(\bm{b}') \rangle.$$ Therefore, we have$$\langle \bm{b}', \bm{\lambda}^*(\bm{b}') \rangle-\langle \bm{b}, \bm{\lambda}^*(\bm{b}) \rangle= \langle \bm{c}, \bm{x}^*(\bm{b}')-\bm{x}^*(\bm{b}) \rangle=f(\bm{b}')-f(\bm{b}). $$ Since $\bm{\lambda}^*(\bm{b})$ and $\bm{\lambda}^*(\bm{b}')$ are optimal vertices of \cref{Intro:dual} and \cref{Intro:dual_p}, respectively, we obtain
$$\langle \bm{b}'-\bm{b}, \bm{\lambda}^*(\bm{b}) \rangle\leq(f(\bm{b}')-f(\bm{b}))\leq \langle \bm{b}'-\bm{b}, \bm{\lambda}^*(\bm{b}') \rangle. $$
Therefore
$$|f(\bm{b}')-f(\bm{b})|\leq \|\bm{b}'-\bm{b}\|\max_{\bm{\lambda}\in \bm{\Lambda}} \|\bm{\lambda}\|,$$ where $\bm{\Lambda}$ is the set of all vertices of the polytope $\bA^\top \lambda \geq \bc$.
    
\end{proof}
We now turn to proofs of the propositions.
\begin{proof}[Proof of Proposition~\ref{lipschitz_vertex}]
	We follow the same argument as is given in the proof of Proposition 3.7 in \cite{klatt2022limit}.
	If $\bx^*(\bb_1)$ and $\bx^*(\bb_2)$ are both nonempty, then Lemma~\ref{continuity_lemmas}, part 4, implies that $\|\val(\bb_1) - \val(\bb_2)\| \leq C_2\|\bb_1 - \bb_2\|$.
	We then apply the main theorem of \cite{walkup1969lipschitzian} with $K$ being the positive orthant and $\tau(\bx) = (\bA \bx, \langle \bc, \bx \rangle)$.
\end{proof}

\begin{proof}[Proof of Proposition~\ref{no_disappearing}]
	Let $\delta$ be small enough that Lemma~\ref{continuity_lemmas} holds.
   Parts 1 and 2 of that lemma imply that $\emptyset \neq \cI^*(\bb') \subseteq \cI(\bb') \subseteq \cI(\bb)$.
   It therefore suffices to show that $\I_0 \in \cI^*(\bb)$ for all $\I_0 \in \cI^*(\bb')$.
   
   Assume that $\bx(\I_0; \bb') \in \bx^*(\bb')$.
   Denote by $\lambda_{\I_0}$ an optimal dual solution to \cref{Intro:dual_p}, which satisfies
   $$\bA^T\bm{\lambda}_{I_0}+\bm{s}=\bc,\ \bA\bx(I_0;\bb')=\bb',\ \bx(I_0;\bb')\geq 0,\ \bm{s}\geq 0, \ \bx(I_0;\bb')^T\bm{s}=0$$
for some $\bm{s} \in \RR^m$.
We will now show that $(\bx(\I_0; \bb), \lambda_{\I_0})$ is also an optimal primal-dual pair for the unperturbed program when $\delta$ is small enough.

The first four conditions still hold for  $(\bx(I_0;\bb),\bm{\lambda}_{I_0})$:
    $$\bA^T\bm{\lambda}_{I_0}+\bm{s}=\bc,\ \bA\bx(I_0;\bb)=\bb,\ \bx(I_0;\bb)\geq 0,\ \bm{s}\geq 0.$$ 
    To show the complementary slackness condition holds, we use Part 3 of Lemma~\ref{continuity_lemmas}, since $S(\bx(I_0;\bb))\subseteq S(\bx(I_0;\bb'))$ as long as $\|\bx(I_0;\bb) - \bx(I_0;\bb')\| < \tau(\bA, \bb)$, where $$\tau(\bA, \bb) := {\max_{\I: \I \in \cI(\bb)} \min_{i\in S(\bx(\I; \bb))} \bx(\I; \bb)_i} > 0.$$
    By Part 3 of Lemma~\ref{continuity_lemmas}, we can choose $\delta'(\bA, \bb) > 0$ small enough that $\|\bx(I_0;\bb) - \bx(I_0;\bb')\| < \tau(\bA,\bb)$ whenever $\|\bb - \bb'\| \leq \delta'$.

We obtain that if $\|\bb' - \bb'\| \leq \delta^*(\bA, \bb) =: \delta \wedge \delta'$, then $\mathcal{I}^*(\bb')\subseteq\mathcal{I}^*(\bb)$, as desired.

\end{proof}

\section{Proofs of main theorems}
This section contains the proofs of our main results.
We first show how to derive \cref{thm:non-unique} and Corollary~\ref{non_unique_cor} (\cref{non-unique-proofs}).
We then obtain \cref{thm:unique} and Corollary~\ref{unique_cor} as easy consequences (\cref{unique-proofs}).
Finally, we give the elementary proofs of \cref{confidence_set} and Corollary~\ref{confidence_cor} in \cref{confidence_proofs}.
\subsection{Proofs for \cref{non-unique}}\label{non-unique-proofs}

Our proof is based on the Hadamard differentiability properties of the mapping $H: \RR^k \to \cC(\mathbb S^{m-1})$ which sends a vector $\bb$ to the support function $\supp_{\bx^*(\bb)}$.
Specifically, we will show the following:
\begin{theorem}\label{hadamard}
	The mapping $H: \RR^k \to \cC(\mathbb S^{m-1})$ is directionally Hadamard differentiable, with derivative $\limfunc{\cdot}$, where $\limfunc{}$ is as in the statement of \cref{thm:non-unique}.
	That is,
\begin{equation}
	\lim_{t_n \searrow 0, \xi_n \to \xi} \frac{H(\bb + t_n \xi_n) - H(\bb)}{t_n} = \limfunc{\xi}\,,
\end{equation}
in $\cC(\mathbb S^{m-1})$.
\end{theorem}
\Cref{thm:non-unique} then follows directly from \cite{romisch2004delta}.

\begin{proof}[Proof of \cref{hadamard}]
First, Proposition~\ref{lipschitz_vertex} and \cref{isometry} imply that if $t_n$ is sufficiently small and $\xi_n$ is sufficiently close to $\xi$, then $$\|\supp_{\bx^*(\bb + t_n \xi_n)} - \supp_{\bx^*(\bb + t_n \xi)}\|_{L^\infty} = \haus(\bx^*(\bb + t_n \xi_n),\bx^*(\bb + t_n \xi)) \lesssim t_n \|\xi_n - \xi\|\,.$$
Therefore
\begin{equation}
	\lim_{t_n \to 0, \xi_n \to \xi} \frac{\|H(\bb + t_n \xi_n) - H(\bb + t_n \xi)\|_{L^\infty}}{t_n} \lesssim 	\lim_{\xi_n \to \xi} \|\xi_n - \xi\| = 0\,,  
\end{equation}
so that
\begin{equation}
	\lim_{t_n \searrow 0, \xi_n \to \xi} \frac{H(\bb + t_n \xi_n) - H(\bb + t_n \xi)}{t_n} = 0
\end{equation}
in $\cC(\mathbb S^{m-1})$.

It therefore suffices to show that 
\begin{equation}
	\lim_{t_n \searrow 0} \frac{H(\bb + t_n \xi) - H(\bb)}{t_n} = \limfunc{\xi}\,.
\end{equation}
The function $\supp(\alpha)$ is differentiable whenever $\sup_{\bx \in \optvert(\bb)} \langle \alpha, \bx \rangle$ is uniquely achieved, and the gradient is precisely the vertex giving the supremum.
For a vertex $\bm{v} \in \optvert(\bb)$ we write $K_\bm{v}$ for the subset of $\mathbb{S}^{m-1}$ consisting of all $\alpha$ for which $\supp$ is differentiable at $\alpha$, with derivative $\bm{v}$.
The collection $\{K_{\bm{v}}: \bm{v} \in \optvert(\bb)\}$ forms a finite disjoint partition of the sphere up to a measure zero set.
We shall show that $H(\bb + t_n \xi)$ converges uniformly to $H(\bb)$ on each element of this partition, which establishes almost everywhere uniform convergence and the desired limit.

In what follows, we therefore fix a $\bm{v} \in \optvert(\bb)$ and consider the functions $H(\bb + t_n \xi)$ and $H(\bb)$ on $K_\bm{v}$.
By assumption, $\sup_{\bx \in \optvert(\bb)} \langle \alpha, \bx \rangle$ is uniquely attained at $\bm{v}$ for all $\alpha$ in this set.
We will now show that for all $\alpha \in K_{\bm{v}}$ and $t_n$ smaller than a constant that depends on $\bm{v}$ but not on $\alpha$, we may restrict the supremum in $\sup_{\bx \in \bx^*(\bb + t_n \xi)} \langle \alpha, \bx \rangle$ to vectors of the form $\bx(\I; \bb + t_n \xi)$ where $\I \in \cI^*(\bb + t_n \xi)$ and $\bx(\I; \bb) = \bm{v}$.

The optimal set $\bx \in \bx^*(\bb + t_n \xi)$ is the set of nonnegative vectors in $\RR^m$ that satisfy the linear constraint $\bA \bx = \bb + t_n \xi$ and that achieve the value  $\langle \bc, \bx \rangle = \val(\bb + t_n \xi)$.
Therefore $\sup_{\bx \in \bx^*(\bb + t_n \xi)} \langle \alpha, \bx \rangle$ is equivalent to the linear program
\begin{equation}\label{supp_lp_bn}
	\max \langle \alpha, \bm{x} \rangle : \bA \bm{x} = \bb+t_n\xi, \langle \bc, \bx \rangle = \val(\bb + t_n \xi), \bx \geq 0\,.
\end{equation}
Analogously, we have by assumption that $\bm{v}$ is the unique solution to
\begin{equation}\label{supp_lp}
	\max \langle \alpha, \bm{x} \rangle : \bA \bm{x} = \bb, \langle \bc, \bx \rangle = \val(\bb), \bx \geq 0\,.
\end{equation}
Since $\bx^*(\bb + t_n \xi)$ is compact, \cref{supp_lp_bn} has an optimal solution, and therefore so does its dual problem:
\begin{equation}\label{supp_dual}
	\min \langle \lambda, \bb + t_n \xi \rangle + \mu \val(\bb + t_n \xi): \bA^\top \lambda + \mu \bc \geq \alpha\,,
\end{equation}
where $\lambda \in \RR^{k}$ and $\mu \in \RR$.
Denote by $\lambda^*$ and $\mu^*$ arbitrary optimal solutions to this problem.
Complementary slackness implies that any optimal solution $\bx^*_n$ to \cref{supp_lp_bn} satisfies
\begin{equation}\label{cs}
	i \in S(\bx^*_n) \implies (\bA^\top \lambda^*)_i+ \mu ^*\bc_i = \alpha_i\,.
\end{equation}

We can always assume that $\sup_{\bx \in \bx^*(\bb + t_n \xi)} \langle \alpha, \bx \rangle$ is achieved at an extreme point, and so is given by some basic feasible solution $\bx(\I; \bb + t_n \xi)$ for $\I \in \cI^*(\bb + t_n \xi)$.
So it suffices to show that if such an $\I$ gives rise to an optimal solution to \cref{supp_lp_bn}, then $\bx(\I; \bb) = \bm{v}$.
By \cref{cs}, 
\begin{equation}\label{cs'}
	(\bA^\top \lambda^*)_i+ \mu ^*\bc_i = \alpha_i \quad \forall i \in S(\bx(\I; \bb + t_n \xi))\,.
\end{equation}
By Proposition~\ref{no_disappearing}, for $t_n$ small enough (independent of $\alpha$), the fact that $\I \in \cI^*(\bb + t_n \xi)$ implies $\I \in \cI^*(\bb)$ and $S(\bx(\I; \bb)) \subseteq S(\bx(\I; \bb+ t_n \xi))$.
Combining this fact with \cref{cs'} gives that 
\begin{equation}\label{cs''}
	\langle \bA^\top \lambda^*+ \mu ^*\bc - \alpha,  \bx(\I; \bb) \rangle = 0\,.
\end{equation}
But this implies that $\bx(\I; \bb)$ must be optimal for \cref{supp_lp} by weak duality.
To see this explicitly, we first observe that $\I \in \cI^*(\bb)$ shows that $\bx(\I; \bb)$ is feasible in \cref{supp_lp}.
Second, $\lambda^*$ and $\mu^*$ are feasible for \cref{supp_dual}.
Therefore, if $\bx$ is any feasible point for \cref{supp_dual}, we have
\begin{align*}
	\langle \alpha, \bx(\I; \bb) - \bx \rangle & = \langle \bA^\top \lambda + \mu \bc, \bx(\I; \bb) - \bx \rangle + \langle \bA^\top \lambda + \mu \bc - \alpha, \bx -\bx(\I; \bb) \rangle \\
	& \geq 0\,,
\end{align*}
where we have used that $\langle \bA^\top \lambda + \mu \bc, \bx(\I; \bb) - \bx \rangle = 0$ since $\bA \bx(\I; \bb) = \bA \bx$ and $\langle \bc, \bx(\I; \bb) \rangle = \langle \bc, \bx \rangle$ and the second term is nonnegative in light of \cref{cs''} and the fact that $\bA^\top \lambda + \mu \bc - \alpha$ and $\bx$ are both nonnegative.
Therefore $\bx(\I;\bb)$ is optimal for \cref{supp_lp}, so we must have $\bx(\I;\bb) = \bm{v}$, which was what we wanted to show.

For $\alpha \in K_{\bm{v}}$, we therefore have that for $t_n$ sufficiently small, depending only on $\bm{v}$,
\begin{equation}\label{find_bases}
	\sup_{\bm{x} \in \optvert(\bb + t_n \xi)} \langle \alpha, \bm{x} \rangle - 	\sup_{\bm{x} \in \optvert(\bb)} \langle \alpha, \bm{x} \rangle = \max_{\I \in \cI^*(\bb + t_n \xi): \bx(\I; \bb) = \bm{v}} \langle \alpha, \bx(\I; \bb + t_n \xi) - \bm{v} \rangle\,.
\end{equation}

	Consider now the linear program appearing in the definition of $\bm{q}_\alpha^*(\xi)$:
	\begin{equation}\label{auxiliary}
		\min \langle \bc, \bm{q} \rangle : \bA \bm{q} = \xi, \bm{q}_i \geq 0 \quad \forall i \notin S(\bm{v})\,.
	\end{equation}
	Note that this program does not depend on $\alpha$, only on $\bm{v}$.
	We wish to show that for $t_n$ small enough, the basic feasible solutions of this program are exactly the vectors of the form $t_n^{-1}(\bx(\I; \bb + t_n \xi) - \bm{v})$ for $\I \in \cI(\bb + t_n \xi)$ such that $\bx(\I;  \bb) = \bm{v}$.
	A basic feasible solution corresponds to a selection of $m$ linearly independent constraints: $k$ that arise from the equality constraints, and $m-k$ tight inequality constraints selected from the set $S(\bm{v})^C$.
	
	Fix a basis for this linear program, denote by $\bar J$ the set of equality constraints selected from the set $S(\bm{v})^C$, and let $J = \bar J^C$.
	The fact that $\bar J \subseteq S(\bm{v})^C$ implies $S(\bm{v}) \subseteq J$.
	Since these give rise to a basis, the system of equations given by $\bA \bm{q} = \xi$ and $\bm{q}_j = 0$ for $j \notin J$ has a unique solution.
	Equivalently, the set $J$ satisfies that $\bA_{J} \bm{q}_J = \xi$ has a unique solution, so that $\bA_{J}$ is full rank.
	If this basis gives rise to a basic feasible solution of \cref{auxiliary}, then $(\bA_J^{-1} \xi)_i \geq 0$ for all $i \notin S(\bm{v})$.
	To conclude, basic feasible solutions to \cref{auxiliary} are of the form $\bm{q}_J = \bA_J^{-1} \xi$ and $\bm{q}_{J^C}= \bz$, where $J \supseteq S(\bm{v})$ satisfies
	\begin{equation}\label{req_1}
		|J| = k, \operatorname{rank}(\bA_J) = k, (\bA_J^{-1} \xi)_i \geq 0 \quad \forall i \notin S(\bm{v})\,.
	\end{equation} 
	Conversely, every set $J$ satisfying these requirements gives rise to a basic feasible solution of \cref{auxiliary}.

	On the other hand, if  $\bm{y} = t_n^{-1}(\bx(\I; \bb + t_n \xi) - \bm{v})$ for some $\I \in \cI(\bb + t_n \xi)$ such that $\bx(\I;  \bb) = \bm{v}$, then $\bm{y}_{\I} = t_n^{-1} \bA_\I^{-1}(\bb + t_n \xi - \bb) = \bA_\I^{-1} \xi$ and $\bm{y}_{\I^C} = \bz$.
	Moreover, the requirement that $\I \in \cI(\bb + t_n \xi)$ implies that $\bm{y}_{\I} = \bA_\I^{-1} \xi \geq - t_n^{-1} \bm{v}_\I$, since this is equivalent to $\I$ being feasible for $\bb + t_n \xi$, and the requirement that $\bx(\I; \bb) = \bm{v}$ implies $S(\bm{v}) \subseteq I$.
	To conclude, $\bm{y}$ is a vector such that $\bm{y}_\I = \bA_\I^{-1} \xi$ and $\bm{y}_{\I^C} = \bz$, where $\I \supseteq S(\bm{v})$ satisfies
	\begin{equation}\label{req_2}
		|\I| = k, \operatorname{rank}(\bA_\I) = k, \bA_\I^{-1} \xi \geq -t_n^{-1} \bm{v}_{\I}\,.
	\end{equation}
	Conversely, any set $\I$ satisfying these properties gives rise to a vector $\bm{y}$ of the form $t_n^{-1}(\bx(\I; \bb + t_n \xi) - \bm{v})$ for some $\I \in \cI(\bb + t_n \xi)$ such that $\bx(\I;  \bb) = \bm{v}$.
	
	We now notice that \cref{req_1} and \cref{req_2} are nearly the same.
	Clearly, all sets $I \supseteq S(\bm{v})$ satisfying \cref{req_2} also satisfy $(\bA_I^{-1} \xi)_i \geq 0 \quad \forall i \notin S(\bm{v})$, since if $i \notin S(\bm{v})$ this is equivalent to the requirement that $(\bA_\I^{-1} \xi)_i \geq -t_n^{-1} \bm{v}_i = 0$ in \cref{req_2}.
	Conversely, for $t_n$ sufficiently small, every set $J\supseteq S(\bm{v})$ satisfying \cref{req_1} also satisfies $\bA_J^{-1} \xi \geq -t_n^{-1} \bm{v}_J$
	This is because every coordinate of the vector $\bA_J^{-1} \xi$ is bounded, uniformly in $J$.
	So for $t_n$ small enough, if $\bm{v}_i > 0$, we will have $(\bA_J^{-1} \xi)_i \geq -t_n^{-1} \bm{v}_i$.
	Therefore, for $t_n$ small enough, the allowable subsets in \cref{req_1} and \cref{req_2} agree.
	In other words, the basic feasible solutions to \cref{auxiliary} are precisely the vectors of the form $t_n^{-1}(\bx(\I; \bb + t_n \xi) - \bm{v})$ for some $\I \in \cI(\bb + t_n \xi)$ such that $\bx(\I;  \bb) = \bm{v}$.
	Moreover, it is now easy to see that \emph{optimal} vertices in \cref{auxiliary} correspond to vectors of the form $t_n^{-1}(\bx(\I; \bb + t_n \xi) - \bm{v})$ for some $\I \in \cI^*(\bb + t_n \xi)$ (i.e., the set of \emph{optimal} bases) for which $\bx(\I;  \bb) = \bm{v}$.
	Indeed, in each case we simply need to select the subset of vertices that minimize the inner product with $\bm{c}$: that is obviously true in the case of solutions to \cref{auxiliary}, and by linearity a basic feasible solution $\bx(\I; \bb + t_n \xi)$ minimizes the inner product with $\bm{c}$ if and only if $t_n^{-1}(\bx(\I; \bb + t_n \xi) - \bm{v})$ minimizes the inner product with $\bm{c}$.
	
	We conclude that for $t_n$ small enough (depending only on $\bm{v}$ and not on $\alpha$), for all $\alpha \in K_{\bm{v}}$,
	\begin{equation}
		\max_{\I \in \cI^*(\bb + t_n \xi): \bx(\I; \bb) = \bm{v}} \langle \alpha, \bx(\I; \bb + t_n \xi) - \bm{v} \rangle = \max_{\bm{x} \in \bm{q}_\alpha^*(\xi)} \langle \alpha, t_n \bm{x} \rangle = t_n \limfunc{\xi}(\alpha)\,.
	\end{equation}
	Therefore
	\begin{equation}
		\lim_{t_n \searrow 0} \frac{H(\bb + t_n \xi)(\alpha) - H(\bb)(\alpha)}{t_n} = \limfunc{\xi}(\alpha)
	\end{equation}
	uniformly on $K_{\bm{v}}$, as claimed.
\end{proof}
\begin{proof}[Proof of Corollary~\ref{non_unique_cor}]
	The functional $f \mapsto \sup_{\alpha \in \mathbb S^{m-1}} f(\alpha)$ is clearly a continuous map from  $\cC( \mathbb S^{m-1})$ to $\RR$, so the continuous mapping theorem combined with \cref{thm:non-unique} implies
	\begin{equation*}
		r_n \sup_{\alpha \in \mathbb S^{m-1}} |\supp_n(\alpha) - \supp(\alpha)| \weakto \sup_{\alpha \in \mathbb S^{m-1}} |\limfunc{\GG}(\alpha)|\,.
	\end{equation*}
	Combined with \cref{isometry}, this implies
	\begin{equation*}
		r_n 	\haus(\bx^*(\bb_n), \bx^*(\bb)) \weakto  \sup_{\alpha \in \mathbb S^{m-1}} |\limfunc{\GG}(\alpha)|\,.
	\end{equation*}
\end{proof}
\subsection{Proofs for \cref{unique}}\label{unique-proofs}
The results of this section will follow from specializing the results of \cref{non-unique} to the case where the target solution $\bx^*(\bb)$ is unique.
\begin{proof}[Proof of \cref{thm:unique}]
	We will apply \cref{thm:non-unique}.
	We first need to verify that the sense of convergence is the same, and then that the expressions for the limit agree.
	The random solution set $\bx^*(\bb_n)$ is nonempty with probability approaching $1$ as $n \to \infty$ by Lemma~\ref{continuity_lemmas}.
	The sets on the left side of the limit in \cref{thm:unique} are therefore (on an event of probability approaching one) non-empty, convex, compact sets. If $K_1, \dots,$ is a sequence of such random sets, \citet[Theorem 6.13]{molchanov2005theory} implies that it converges weakly to a random set $K$ if and only if for any $N \in \NN$, $\alpha_1, \dots, \alpha_N \in \mathbb{S}^{m -1}$, the vector $(\supp_{K_n}(\alpha_1), \dots, \supp_{K_n}(\alpha_N))$ converges to $(\supp_{K}(\alpha_1), \dots, \supp_{K}(\alpha_N))$, and the sets are tight, in the sense that $\lim_{c \to \infty} \sup_n \p{\|K_n\| \geq c} \to 0$.
	
	To compute the support function of $r_n(\bx^*(\bb_n) - \bx^*(\bb))$, we use the fact that $\bx^*(\bb)$ is a singleton to write
	\begin{align*}
		\sup_{\bx \in r_n(\bx^*(\bb_n) - \bx^*(\bb))} \langle \alpha, \bx \rangle & = \sup_{\bx' \in \bx^*(\bb_n)} \langle \alpha, r_n (\bx' - \bx^*(\bb)) \\
		& = r_n \sup_{\bx' \in  \bx^*(\bb_n)} \langle \alpha, \bx' \rangle - r_n \langle \alpha, \bx^*(\bb) \rangle \\
		& = r_n(\supp_n(\alpha) - \supp(\alpha))\,,
	\end{align*}
	where $\supp_n$ and $\supp$ are as in \cref{thm:non-unique}.
	\Cref{thm:non-unique} shows that the support function of $r_n(\bx^*(\bb_n) - \bx^*(\bb))$ converges to $\limfunc{\GG}$.
The tightness condition is therefore trivially satisfied, so we will be done as long as we can show that the function $\limfunc{\GG}$ is the support function of the set $\bm{p}_\bb^*(\GG)$.
	Since $\supp(\alpha) = \langle \alpha, \bx^*(\bb) \rangle$, the gradient $\nabla h(\alpha)$ is identically equal to $\bx^*(\bb)$, so that the linear program \cref{non_unique_lp} reduces to \cref{unique_auxiliary_lp}.
	Since $\limfunc{\GG}$ is defined as the supremum of a linear functional, we may replace the set $\bm{q}^*(\GG)$ of optimal vertices by its convex hull $\operatorname{conv}(\bm{q}^*(\GG))$, and we just need to show that this set agrees with $\bm{p}_\bb^*(\GG)$ to show that $\limfunc{\GG}$ is its support function.
	To do so, we use the fact that the recession cone of $\bm{p}_\bb^*(\GG)$ is $\{\bz\}$ when $\bx^*(\bb)$ is unique.
	To see this, first observe that a vector in the recession cone must satisfy $\langle \bc, \bm{d} \rangle = 0, \bA \bm{d} = 0, \bm{d}_i \geq 0$ for all $i \notin S(\bx^*(\bb))$.
	If $\bm{d}$ is such a vector, then for $\epsilon > 0$ small enough the vector $\bx^*(\bb) + \epsilon \bm{d}$ is also optimal for \cref{primal}.
	Indeed this vector satisfies the linear constraints and has the same objective value, and for $\epsilon$ sufficiently small no coordinates of $\bx^*(\bb) + \epsilon \bm{d}$ will be negative.
	Since we have assumed that $\bx^*(\bb)$ is a singleton, we must have that $\bm{d} = \bz$, so that the recession cone of the optimal set in this LP is $\{\bz\}$.
	Therefore $\bm{p}_\bb^*(\GG) = \operatorname{conv}(\bm{q}^*(\GG))$, and therefore $\limfunc{\GG}$ is the support function of $\bm{p}_\bb^*(\GG)$, proving the claim.
\end{proof}

\begin{proof}[Proof of Corollary \ref{unique_cor}]
	For any vector $\bx$, the functional $S \mapsto d(S, \bx)$ is continuous with respect to the Hausdorff distance.
	The continuous mapping theorem therefore implies that
	\begin{equation*}
		d(r_n (\bx^*(\bb_n) - \bx^*(\bb)), \bz) \weakto d(\bm{p}_\bb^*(\GG), \bz)\,.
	\end{equation*}
	It then suffices to note that the quantity on the left is equal to $r_n d(\bx^*(\bb_n),  \bx^*(\bb))$.
\end{proof}
\subsection{Proofs for \cref{confidence}}\label{confidence_proofs}

\begin{proof}[Proof of \cref{confidence_set}]
Define $\bm{A}_n\coloneqq (\bm{A};\bm{e}_{\I_n^c}) \in \RR^{m \times m}$ to be the matrix whose first $k$ rows are $\bm{A}$ and whose remaining $m-k$ rows consist of the elementary basis vectors $\bm{e}_i$ for $i \notin \I_n$.
	Since $\I_n$ is a basis, $\bA_n$ has full rank.
	Moreover, the fact that the basis $\I_n$ corresponds to $\hat \bx_n$ implies that $S(\hat \bx_n) \subseteq \I_n$.
	Therefore $\bA_n \hat \bx_n = \bb^0_n$, where $\bb^0_n \defeq (\bb_n, \bz) \in \RR^m$ is the augmented vector whose first $k$ coordinates are $\bb_n$ and whose remaining $m - k$ coordinates are zero.
	Similarly, $\bA_n \bx(\I_n; \bb) = \bb^0$, where $\bb^0 \in \RR^m$ is defined in an analogous way.
	
We obtain
	\begin{equation}
	r_n \bm{A}_n(\hat{\bm{x}}_n-\bm{x}(\I_n;\bm{b}))=r_n(\bm{b}^0_n-\bm{b}^0)\xrightarrow[]{D}\mathbb{G}^0,
\label{confidence set: conv dist part}
	\end{equation}
where as above $\mathbb{G}^0$ is the random variable obtained by appending $m -k$ zeroes to $\mathbb G$.
	
	We will now show that $r_n(\proj - \bx(\I_n; \bb)) \overset{p}{\to} \bz$, so that we can replace $\bx(\I_n; \bb)$ by $\proj$ in the limit.
	By \cref{no_disappearing}, there exists a constant $\delta > 0$ such that if $\|\bb_n - \bb\| \leq \delta$, then $\cI^*(\bb_n) \subseteq \cI^*(\bb)$.
	Since $\I_n \in \cI^*(\bb_n)$ by assumption, this fact implies that if $\|\bb_n - \bb\| \leq \delta$, then $\I_n$ is an optimal basis for the target problem, i.e., $\bx(\I_n; \bb) \in \bx^*(\bb)$.
	In particular, on the event that $\|\bb_n - \bb\| \leq \delta$, we have $\proj = \bx(\I_n; \bb)$.
	The distributional convergence assumption \cref{b_limit} implies $\bb_n \overset{p}{\to} \bb$.
	We therefore have that $\p{r_n \|\proj - \bx(\I_n; \bb)\| > 0 } \leq \p{\|\bb_n - \bb\| > \delta} \to 0$ as $n \to \infty$, so that $r_n(\proj - \bx(\I_n; \bb)) \overset{p}{\to} \bz$.
	Combining this fact with \cref{confidence set: conv dist part} yields
	\begin{equation}
		r_n\bm{A}_n(\hat{\bm{x}}_n-\proj)\xrightarrow[]{D}\mathbb{G}^0,
		\label{confidence set: conv final}
	\end{equation}

	If we define $G_\alpha$ as in the theorem, we therefore obtain that
	\begin{equation}
		\liminf_{n \to \infty} \p{r_n(\hat{\bm{x}}_n-\proj) \in \bm{A}_n^{-1} G^0_\alpha} = 		\liminf_{n \to \infty} \p{r_n\bm{A}_n(\hat{\bm{x}}_n-\proj) \in G^0_\alpha} \geq 1 - \alpha
	\end{equation}
	where $G^0_\alpha \in \RR^{m}$ is obtained from $G_\alpha \in \RR^k$ by padding each vector with zeros.
	To conclude, we note that $\bm{A}_n^{-1} G^0_\alpha = \bx(\I_n; G_\alpha)$.
	Indeed, for any $\bm{G} \in G_\alpha$, the definition of $\bm{A}_n$ implies that the first $k$ coordinates of $\bm{A}_n \bx(\I_n; \bm{G})$ are $\bm{A} \bx(\I_n; \bm{G}) = \bm{G}$ and the last $m-k$ coordinates are zero.
	Therefore $G^0_\alpha = \bm{A}_n \bx(\I_n; G_\alpha)$, and since $\bm{A}_n$ is invertible this proves the claim.
\end{proof}
Corollary~\ref{confidence_cor} is an immediate consequence.
\begin{proof}[Proof of Corollary~\ref{confidence_cor}]
	If $r_n(\hat \bx_n - \proj) \in \bx(\I_n; G_\alpha)$, then there exists an $\bx \in \bx(\I_n; G_\alpha)$ such that $\proj = \hat \bx_n - r_n^{-1} \bx$.
	Since $\proj \in \bx^*(\bb)$, the result follows from \cref{confidence_set}.
\end{proof}
\bibliography{ref}

\end{document}